\documentclass[smallextended]{svjour3}

\smartqed
\usepackage{graphicx}
\usepackage{geometry}                
\usepackage{graphicx}
\usepackage{amssymb}
\usepackage{epstopdf}
\usepackage{pgf,tikz}
\usepackage{latexsym}
\usepackage{amsmath,amsfonts}
\usepackage{graphicx}
\usepackage{geometry}
\usepackage{graphicx}
\usepackage{amssymb}
\usepackage{epstopdf}
\usepackage{pgf,tikz}
\usepackage{latexsym}

\newcommand{\ignore}[1]{}

\newcommand{\oR}{{\mathbb R}}
\newcommand{\oN}{{\mathbb N}}

\journalname{}
\begin{document}

\title{A refined error analysis for fixed-degree polynomial optimization over the simplex}

\author{Zhao Sun}
\institute{Tilburg University \at
              PO Box 90153, 5000 LE Tilburg \\
              Tel.: +31-13-4663313\\
              Fax:  +31-13-4663280\\
              \email{z.sun@uvt.nl}
}

\date{Received: date / Accepted: date}

\maketitle

\begin{abstract}
We consider the problem of minimizing a fixed-degree polynomial  over the standard simplex. This problem is well known to be NP-hard, since it contains the maximum stable set problem in combinatorial optimization as a special case. In this paper, we revisit  a known upper bound obtained by taking the minimum value on a regular grid, and a known lower bound based on P\'olya's representation theorem. More precisely, we consider  the difference between these two bounds and we provide upper bounds for this difference in terms of the range of function values. Our results refine the known upper bounds in  the quadratic and cubic cases, and they asymptotically refine the known upper bound in the general case.
\end{abstract}

\section{Introduction and preliminaries}\label{secintro}

\noindent Consider the problem of minimizing a homogeneous polynomial $f\in\oR[x]$ of degree $d$ on the (standard) simplex
$$\Delta_n:=\{x\in\oR_+^n:\sum_{i=1}^nx_i=1\}.$$
That is, the global optimization problem:
\begin{equation}\label{oriprob}
\underline{f}:=\min_{x\in\Delta_n}f(x), \ \ \text{or}\ \ \overline{f}:=\max_{x\in\Delta_n}f(x).
\end{equation}
Here we focus on the problem of computing the minimum $\underline f$ of $f$ over $\Delta_n$.
This problem is well known to be NP-hard, as it contains the maximum stable set problem as a special case (when $f$ is quadratic). Indeed, given a graph $G=(V,E)$ with adjacency matrix $A$, Motzkin and Straus \cite{MS} show that  the maximum stability number $\alpha(G)$ can be obtained by
$${1\over \alpha(G)}=\min_{x\in\Delta_{|V|}}x^T(I+A)x,$$
where  $I$ denotes the identity matrix. Moreover, one can w.l.o.g. assume $f$ is homogeneous. Indeed, if $f=\sum_{s=0}^{d}f_s$, where $f_s$ is homogeneous of degree $s$, then $\min_{x\in\Delta_n}f(x)=\min_{x\in\Delta_n}f'(x)$, setting $f'=\sum_{s=0}^d f_s\left(\sum_{i=1}^n x_i\right)^{d-s}.$\\

\noindent For problem (\ref{oriprob}), many approximation algorithms have been studied in the literature. In fact, when  $f$ has fixed degree $d$, there is a polynomial time approximation scheme (PTAS) for this problem, see  \cite{BK02} for the case $d=2$ and \cite{KLP06,KLS13} for  $d\ge 2$. For more results on its computational complexity, we refer to \cite{EDK08,KHE08}.\\

\noindent
We consider the following two bounds for $\underline{f}$: an upper bound $f_{\Delta(n,r)}$ obtained by taking the minimum value on a regular grid and a lower bound $f_{\min}^{(r-d)}$ based on P\'olya's representation theorem. They both have been studied in the literature, see e.g. \cite{BK02,KLP06,KLS13} for $f_{\Delta(n,r)}$ and \cite{KLP06,SY13,YEA12} for $f_{\min}^{(r-d)}$.
The two ranges $f_{\Delta(n,r)}-\underline f$ and $\underline f- f_{\min}^{(r-d)}$ have been  studied  separately and upper bounds for each of  them have been
shown in the above mentioned works.\\

\noindent In this paper, we study these two ranges at the same time. More precisely,  we analyze the  larger range $f_{\Delta(n,r)}-f_{\min}^{(r-d)}$ 
and provide upper bounds for it in terms of the range of function values $\overline{f}-\underline{f}$.
Of course, upper bounds for the range $f_{\Delta(n,r)}-f_{\min}^{(r-d)}$ can be obtained by combining the known upper bounds for each of the two ranges $f_{\Delta(n,r)}-\underline f$ and $\underline f- f_{\min}^{(r-d)}$.
Our new upper bound for  $f_{\Delta(n,r)}-f_{\min}^{(r-d)}$ refines these known bounds in the quadratic and cubic cases  and provide an asymptotic refinement for general degree $d$.


\subsubsection*{Notation}
\noindent Throughout $\mathcal{H}_{n,d}$ denotes the set of all homogeneous polynomials in $n$ variables with degree $d$. We let $[n]:=\{1,2,\ldots,n\}$. We denote $\oR^n_+$ as the set of all nonnegative real vectors, and $\oN^n$ as the set of all nonnegative integer vectors. For $\alpha\in\oN^n$, we define $|\alpha|:=\sum_{i=1}^n\alpha_i$ and $\alpha!:=\alpha_1!\alpha_2!\cdots\alpha_n!$. We denote $I(n,d):=\{\alpha\in \oN^n: |\alpha|=d\}$. We let $e$ denote the all-ones vector and $e_i$ denote the $i$-th standard unit vector. We denote $\oR[x]$ as the set of all multivariate polynomials in $n$ variables (i.e. $x_1,x_2\dots,x_n$) and denote $\mathcal{H}_{n,d}$ as the set of all multivariate homogeneous polynomials in $n$ variables with degree $d$. For $\alpha\in \oN^n$, we denote $x^{\alpha}:=\prod_{i=1}^nx_i^{\alpha_i}$, while for $I\subseteq[n]$, we let $x^{I}:=\prod_{i\in I}x_i$. Moreover, we denote $x^{\underline{d}}:=x(x-1)(x-2)\cdots (x-d+1)$ for integer $d\ge0$ and $x^{\underline{\alpha}}:=\prod_{i=1}^nx_{i}^{\underline{{\alpha}_i}}$ for $\alpha\in \oN^n$.
Thus, $x^{\underline d}=0$ if $x$ is an integer with $0\le x\le d-1$.

\subsubsection*{Upper bounds using regular grids}

\noindent One can construct an upper bound for $\underline{f}$ by taking the minimum of $f$ on the regular grid
$$\Delta(n,r):=\{x\in\Delta_n:rx\in\oN^n\},$$
for an integer $r\ge0$. We define
\begin{equation*}
f_{\Delta(n,r)}:=\min_{x\in\Delta(n,r)}f(x).
\end{equation*}

\noindent Obviously, $\underline{f}\le f_{\Delta(n,r)}\le \overline{f}$, and $f_{\Delta(n,r)}$ can be computed by $|\Delta(n,r)| = {n+r-1 \choose r}$ evaluations of $f$.  In fact, when considering polynomials $f$ of fixed degree $d$, the parameters $f_{\Delta(n,r)}$ (with increasing values of $r$) provide  a PTAS for (\ref{oriprob}), as was proved by Bomze and de Klerk \cite{BK02} (for $d=2$), and by de Klerk et al. \cite{KLP06} (for  $d\ge 2$). Recently, de Klerk et al. \cite{KLS13} provide an alternative proof for this PTAS and refine the error bound for $f_{\Delta(n,r)}-\underline{f}$ from \cite{KLP06} for cubic $f$.

\noindent In addition, some researchers study the properties of the regular grid $\Delta(n,r)$. For instance, given a point $x \in \Delta_n$, Bomze et al. \cite{BGY} show a scheme to find the closest point to $x$ on $\Delta(n,r)$ with respect to some class of norms including $\ell_p$-norms for $p\ge 1$.

\subsubsection*{Lower bounds based on P\'olya's representation theorem}

\noindent
Given a polynomial $f\in\mathcal{H}_{n,d}$,
P\'olya \cite{Pol74} shows that if $f$  is positive over the simplex $\Delta_n$,
then the polynomial $(\sum_{i=1}^ nx_i)^r f$ has nonnegative coefficients for any $r$ large enough (see \cite{PR01} for an explicit  bound for $r$). Based on this result of P\'olya, 
an asymptotically converging hierarchy of lower bounds for $\underline{f}$ can be constructed as follows: for any integer $r\ge d$, we define the parameter $f_{\min}^{(r-d)}$ as
\begin{equation}\label{fminr}
f_{\min}^{(r-d)}:=\max \lambda \ \ \text{s.t.}\ \ \left(\sum_{i=1}^nx_i\right)^{r-d}\left(f-\lambda\left(\sum_{i=1}^nx_i\right)^d\right)\ \
\text{has nonnegative coefficients.}
\end{equation}

\noindent Notice that $\underline{f}$ can be equivalently formulated as
\begin{equation*}
\underline{f}=\max\ \ \lambda \ \ \text{s.t.}\ \ f(x)-\lambda\left(\sum_{i=1}^nx_i\right)^d\ge 0\ \ \forall x\in\oR^n_+.
\end{equation*}

\noindent Then, one can easily check the following inequalities:
\begin{equation*}\label{inequ1}
f_{\min}^{(0)}\le f_{\min}^{(1)}\le\cdots \le\underline{f}\le f_{\Delta(n,r)}\le \overline{f}.
\end{equation*}
Parrilo \cite{Par00,Par03} first introduces the idea of applying P\'olya's representation theorem to construct hierarchical approximations in copositive optimization. De Klerk et al. \cite{KLP06} consider $f_{\min}^{(r-d)}$ and show upper bounds for $\underline{f}-f_{\min}^{(r-d)}$ in terms of $\overline{f}-\underline{f}$. Furthermore, Yildirim \cite{YEA12} and Sagol and Yildirim \cite{SY13} analyze error bounds for $f_{\min}^{(r-2)}$ for quadratic $f$.

\noindent Now we give an explicit formula for the parameter $f_{\min}^{(r-d)}$,  which follows from  \cite[relation (3)]{PR01}; note that  the quadratic case of this formula has also been observed in \cite{PVZ07,SY13,YEA12}.

\begin{lemma}\label{lemrf}
For $f=\sum_{\beta\in I(n,d)}f_{\beta}x^{\beta}\in\mathcal{H}_{n,d}$, one has
\begin{equation}\label{fminrd}
f_{\min}^{(r-d)}=\min_{\alpha\in I(n,r)}\sum_{\beta\in I(n,d)}f_{\beta}{ \alpha^{\underline{\beta}} \over r^{\underline{d}}}.
\end{equation}
 \end{lemma}

\begin{proof}
By using  the multinomial theorem $(\sum_{i=1}^n x_i)^d=\sum_{\alpha\in I(n,d)}{d!\over \alpha!}x^{\alpha}$, we obtain
\begin{eqnarray*}
\left(\sum_{i=1}^nx_i\right)^{r-d}f-\lambda \left(\sum_{i=1}^nx_i\right)^{r}&=&
\left(\sum_{\gamma\in I(n,r-d)}{(r-d)!\over \gamma!}x^{\gamma}\right)\left(\sum_{\beta\in I(n,d)}f_{\beta}x^{\beta}\right)-\lambda\left(\sum_{\alpha\in I(n,r)}{r!\over\alpha!}x^{\alpha}\right)\\
&=&\sum_{\alpha\in I(n,r)}\left(\sum_{\beta\in I(n,d)}f_{\beta}\alpha^{\underline{\beta}}{1\over r^{\underline{d}}}\right){r!\over\alpha!}x^{\alpha}-\lambda\left(\sum_{\alpha\in I(n,r)}{r!\over\alpha!}x^{\alpha}\right)\\
&=&\sum_{\alpha\in I(n,r)}\left(\sum_{\beta\in I(n,d)}f_{\beta}\alpha^{\underline{\beta}}{1\over r^{\underline{d}}}-\lambda\right){r!\over\alpha!}x^{\alpha}.
\end{eqnarray*}
\noindent Hence, by definition (\ref{fminr}), we obtain
\begin{eqnarray*}
f_{\min}^{(r-d)}&=&\max\ \ \lambda \ \ \text{s.t} \ \ \sum_{\beta\in I(n,d)}f_{\beta}\alpha^{\underline{\beta}}{1\over r^{\underline{d}}}-\lambda\ge 0\ \ \forall \alpha\in I(n,r)\\
&=& \min \sum_{\beta\in I(n,d)}f_{\beta}\alpha^{\underline{\beta}}{1\over r^{\underline{d}}}\ \ \text{s.t}\ \ \alpha\in I(n,r).
\end{eqnarray*}
\qed
\end{proof}

\noindent Similarly as $f_{\Delta(n,r)}$, by (\ref{fminrd}), the computation of $f_{\min}^{(r-d)}$ requires $|I(n,r)| = {n+r-1 \choose r}$ evaluations of the polynomial $\sum_{\beta\in I(n,d)}f_{\beta}\alpha^{\underline{\beta}}{1\over r^{\underline{d}}}$.

\subsubsection*{Bernstein coefficients}
\noindent For any  polynomial $f=\sum_{\beta\in I(n,d)}f_{\beta}x^{\beta}\in\mathcal{H}_{n,d}$, we can write it as
\begin{equation}\label{berncoef}
f=\sum_{\beta\in I(n,d)}f_{\beta}x^{\beta}=\sum_{\beta\in I(n,d)}\left(f_{\beta}{\beta!\over d!}\right){d!\over \beta!}x^{\beta}.
\end{equation}

\noindent For any $\beta\in I(n,d)$, we call $f_{\beta}{\beta!\over d!}$ the {\em Bernstein coefficients} of $f$ (this terminology has also been used in \cite{KL10,KLS13}), since they are the coefficients of the polynomial $f$ when $f$ is expressed in the Bernstein basis $\{{d!\over \beta!}x^\beta: \beta\in I(n,d)\}$ of $\mathcal{H}_{n,d}$. Applying the multinomial theorem together with (\ref{berncoef}), one can obtain that  when evaluating $f$ at a point $x\in \Delta_n$, $f(x)$ is a convex combination of
the  Bernstein coefficients $f_{\beta}{\beta!\over d!}$. Therefore, we have
\begin{equation}\label{reprop1}
\min_{\beta\in I(n,d)}f_{\beta}{\beta!\over d!}\le\underline{f}\le f_{\Delta(n,r)}\le\overline{f}\le\max_{\beta\in I(n,d)}f_{\beta}{\beta!\over d!}.
\end{equation}

\noindent For the analysis in Section \ref{secgene}, we need the following result of \cite{KLP06}, which bounds the range of the Bernstein coefficients  of $f$ in terms of its range of values $\overline{f}-\underline{f}$.

\begin{theorem}\label{thmgene}\cite[Theorem 2.2]{KLP06}
For any polynomial $f=\sum_{\beta\in I(n,d)}f_{\beta}x^{\beta}\in\mathcal{H}_{n,d}$, one has
\begin{equation*}
\max_{\beta\in I(n,d)}f_{\beta}{\beta!\over d!}-\min_{\beta\in I(n,d)}f_{\beta}{\beta!\over d!}\le {2d-1\choose d}d^d (\overline{f}-\underline{f}).
\end{equation*}
\end{theorem}

\subsubsection*{Contribution of the paper}
\noindent In this paper, we consider upper bounds for $f_{\Delta(n,r)}-f_{\min}^{(r-d)}$ in terms of $\overline{f}-\underline{f}$. More precisely, we provide 
tighter  upper bounds in  the quadratic, cubic, and square-free (aka multilinear) cases and,  in the general case $d\ge 2$,
our upper bounds are asymptotically tighter when $r$ is large enough.   We will apply the formula (\ref{fminrd}) directly for the quadratic, cubic and square-free cases, while for the general case we will use Theorem \ref{thmgene}.

\noindent There are some relevant results in the literature. De Klerk et al. \cite{KLP06} give upper bounds for $f_{\Delta(n,r)}-\underline{f}$ (the upper bound for cubic $f$ has been refined by de Klerk et al. \cite{KLS13}) and for $\underline{f}-f_{\min}^{(r-d)}$ in terms of $\overline{f}-\underline{f}$, and by adding them up one can easily derive upper bounds for $f_{\Delta(n,r)}-f_{\min}^{(r-d)}$. Furthermore, for quadratic polynomial $f$, Yildirim \cite{YEA12} considers the  upper bound $\min_{k\le r} f_{\Delta(n,k)}$ for $\underline{f}$ (for  $r \ge 2$)  and  upper bounds the range   $\min_{k\le r} f_{\Delta(n,k)}-f_{\min}^{(r-d)}$ in terms of $\overline{f}-\underline{f}$. Our results in this paper refine the results in \cite{KLP06,KLS13,YEA12} for the quadratic and cubic cases (see Sections \ref{secdeg2} and \ref{secdeg3} respectively), while for the general case our result refines the result of \cite{KLP06} when $r$ is sufficiently large (see Section~\ref{secgene}).

\subsubsection*{Structure}
\noindent The paper is organized as follows. In Sections \ref{secdeg2} and \ref{secdeg3}, we consider the quadratic and cubic cases respectively, and refine the relevant results obtained from \cite{KLP06,KLS13,YEA12}.
Then, we look at the square-free (aka multilinear) case in Section \ref{secsqfree}. Moreover, in Section \ref{secgene}, we consider general (fixed-degree) polynomials and compare our new result with the one of \cite{KLP06}.

\section{The quadratic case}\label{secdeg2}
\noindent For any quadratic polynomial $f$, we consider the range $f_{\Delta(n,r)}-f_{\min}^{(r-2)}$ and derive the following 
upper bound in terms of $\overline{f}-\underline{f}$.
\begin{theorem}\label{thmqua}
For any quadratic $f=x^TQx$ and $r\ge 2$, one has
\begin{eqnarray}
f_{\Delta(n,r)}-f_{\min}^{(r-2)}\le {1\over r-1}(Q_{\max}-f_{\Delta(n,r)}) \le {1\over r-1}(\overline{f}-\underline{f})\label{eqn01},
\end{eqnarray}
where $Q_{\max}:=\max_{i\in[n]} Q_{ii}$.
\end{theorem}

\begin{proof}
By (\ref{fminrd}), we have $$f_{\min}^{(r-2)}=\min_{\alpha\in I(n,r)}{1\over r(r-1)}\left[f(\alpha)-\sum_{i=1}^nQ_{ii}\alpha_i\right].$$
Hence, ${r-1\over r}f_{\min}^{(r-2)}=\min_{\alpha\in I(n,r)}\left[f({\alpha\over r})-\sum_{i=1}^nQ_{ii}{\alpha_i\over r}{1\over r}\right].$ We obtain
\begin{equation}
{r-1\over r}f_{\min}^{(r-2)}\ge
\min_{\alpha\in I(n,r)}f({\alpha\over r})-\max_{\alpha\in I(n,r)}{1\over r}\sum_{i=1}^nQ_{ii}{\alpha_i\over r}
= f_{\Delta(n,r)}-{1\over r} Q_{\max}.\label{re4}
\end{equation}
One can easily obtain the first inequality in (\ref{eqn01}) by (\ref{re4}). For the second inequality in (\ref{eqn01}), we use the fact that $ Q_{\max} \le \overline{f}$ (since $Q_{ii}=f(e_i)\le \overline{f}$ for $i\in [n]$), as well as the fact that $f_{\Delta(n,r)}\ge\underline{f}$.
\qed
\end{proof}

\noindent Now we point out  that our result (\ref{eqn01}) refines the relevant result of \cite{KLP06}. De Klerk et al. \cite{KLP06} show the following theorem.
\begin{theorem}\label{klpthmdeg2}\cite[Theorem 3.2]{KLP06}
Suppose $f\in \mathcal{H}_{n,2}$ and $r\ge 2$. Then
\begin{equation}\label{re8}
\underline{f}-f_{\min}^{(r-2)}\le {1\over r-1}(\overline{f}-\underline{f}),
\end{equation}
\begin{equation}\label{re21}
f_{\Delta(n,r)}-\underline{f}\le {1\over r}(\overline{f}-\underline{f}).
\end{equation}
\end{theorem}

\noindent By adding up (\ref{re8}) and (\ref{re21}), one gets
$$f_{\Delta(n,r)}-f_{\min}^{(r-2)}\le \left({1\over r-1}+{1\over r}\right)(\overline{f}-\underline{f}),$$
\noindent which is
 implied by
our result (\ref{eqn01}).

\noindent Moreover, in \cite{YEA12}, Yildirim considers one hierarchical upper bound of $\underline{f}$ (when $f$ is quadratic), which is defined by
$\min_{k\le r}f_{\Delta(n,k)}.$
One can easily verify that
\begin{eqnarray*}
f_{\min}^{(r-2)}\le \underline{f} \le \min_{k\le r}f_{\Delta(n,k)} \le f_{\Delta(n,r)}.
\end{eqnarray*}
In \cite[Theorem 4.1]{YEA12}, Yildirim shows $\min_{k\le r} f_{\Delta(n,k)}-f_{\min}^{(r-2)}\le {1\over r-1}(Q_{\max}-\underline{f})$, which can also be easily implied by our result (\ref{eqn01}).

\noindent The following example shows  that the upper bound (\ref{eqn01}) can be tight.
\begin{example}\cite[Example 2]{KLS13}\label{egdeg2}
Consider the quadratic polynomial $f=\sum_{i=1}^n x_i^2$.
As $f$ is convex, one can check that $\underline{f}={1\over n}$ (attained at $x={1\over n}e$) and $\overline{f}=1$ (attained at any standard unit vector). To compute $f_{\Delta(n,r)}$,
we write $r$ as $r=kn+s$, where $k\ge 0$ and $0\le s<n$. Then one can check that
$$f_{\Delta(n,r)}={1\over n} + {1\over r^2}{s(n-s)\over n}.$$
\noindent By (\ref{fminrd}), we have
$$f_{\Delta(n,r)}-f_{\min}^{(r-2)}={1\over r-1}\left(\overline{f}-\underline{f}\right)-{1\over r^2(r-1)}{s(n-s)\over n}.$$
\noindent Hence, for this example, the upper bound (\ref{eqn01}) is tight when $s=0$.
\end{example}

\section{The cubic case}\label{secdeg3}

\noindent For any cubic polynomial $f$, we consider the difference $f_{\Delta(n,r)}-f_{\min}^{(r-3)}$ and show  the following 
result.

\begin{theorem}\label{thmcub}
For any cubic polynomial $f$ and $r\ge 3$, one has
\begin{eqnarray}\label{iqcub}
f_{\Delta(n,r)}-f_{\min}^{(r-3)}\le {4r\over (r-1)(r-2)}(\overline{f}-\underline{f})\label{eqn2}.
\end{eqnarray}
\end{theorem}

\begin{proof}
We can write any cubic polynomial $f$ as
$$f=\sum_{i=1}^nf_ix_i^3+\sum_{i<j}(f_{ij}x_ix_j^2+g_{ij}x_i^2x_j)+\sum_{i<j<k}f_{ijk}x_ix_jx_k.$$

\noindent Then, by (\ref{fminrd}) one can check that
\begin{eqnarray}
&&{(r-1)(r-2)\over r^2}f_{\min}^{(r-3)}\nonumber\\
&=&\min_{\alpha\in I(n,r)}\left\{f({\alpha\over r})-{1\over r^3}\left( 3\sum_{i=1}^n f_i\alpha_i^2-2\sum_{i=1}^n f_i\alpha_i+\sum_{i<j}(f_{ij}+g_{ij})\alpha_i\alpha_j \right)\right\}\nonumber\\
&\ge& f_{\Delta(n,r)}-{1\over r}\max_{\alpha\in I(n,r)}\left\{3\sum_{i=1}^n f_i\left({\alpha_i\over r}\right)^2+\sum_{i<j}(f_{ij}+g_{ij})\left({\alpha_i\over r}\right)\left({\alpha_j\over r}\right)\right\}+{1\over r^2}\min_{\alpha\in I(n,r)}2\sum_{i=1}^n f_i{\alpha_i\over r}\nonumber\\
&\ge& f_{\Delta(n,r)}-{1\over r}\max_{x\in\Delta_n}\left\{3\sum_{i=1}^n f_ix_i^2+\sum_{i<j}(f_{ij}+g_{ij})x_ix_j\right\}+{1\over r^2}\min_{x\in\Delta_n}2\sum_{i=1}^n f_ix_i.\label{eqfinaladd1}
\end{eqnarray}
Evaluating $f$ at $e_i$ and $(e_i+e_j)/2$ yields, respectively, the relations:
\begin{eqnarray}
\underline{f}\le f_i\le\overline{f},\label{re1}\\
f_i+f_j+f_{ij}+g_{ij}\le 8\overline{f}.\label{re2}
\end{eqnarray}
Using (\ref{re2}) and the fact that $\sum_{i=1}^nx_i=1$, one can obtain
\begin{equation}\label{re3}
\sum_{i<j}(f_{ij}+g_{ij})x_ix_j\le\sum_{i<j}(8\overline{f}-f_i-f_j)x_ix_j=8\overline{f}\sum_{i<j}x_ix_j-\sum_{i=1}^n
f_ix_i(1-x_i).
\end{equation}
By (\ref{eqfinaladd1}), (\ref{re1}), (\ref{re3}) and the fact that $\sum_{i=1}^nx_i=1$, one can get
\begin{eqnarray*}
(r-1)(r-2)f_{\min}^{(r-3)}\ge r^2f_{\Delta(n,r)}-4r\overline{f}+(r+2)\min_{x\in\Delta_n}\sum_{i=1}^nf_ix_i\ge r^2f_{\Delta(n,r)}-4r\overline{f}+(r+2)\underline{f}.
\end{eqnarray*}
Hence, one has
\begin{eqnarray*}
(r-1)(r-2)\left(f_{\Delta(n,r)}-f_{\min}^{(r-3)}\right)\le 4r\overline{f}-(3r-2)f_{\Delta(n,r)}-(r+2)\underline{f}\le 4r(\overline{f}-\underline{f}).
\end{eqnarray*}
\qed
\end{proof}

\noindent Now we observe  that our result (\ref{iqcub}) refines the relevant upper bound obtained from \cite{KLP06,KLS13}. De Klerk et al. \cite{KLP06} show  the following result.

\begin{theorem}\label{klpthmdeg3}\cite[Theorem 3.3]{KLP06}
Suppose $f\in \mathcal{H}_{n,3}$ and $r\ge 3$. Then
\begin{equation}\label{re9}
\underline{f}-f_{\min}^{(r-3)}\le {4r\over (r-1)(r-2)}(\overline{f}-\underline{f}),
\end{equation}
\begin{equation}\label{re10}
f_{\Delta(n,r)}-\underline{f}\le {4\over r}(\overline{f}-\underline{f}).
\end{equation}
\end{theorem}

\noindent Recently, De Klerk et al. \cite[Corollary 2 ]{KLS13} refine (\ref{re10}) to
\begin{equation}\label{re11}
f_{\Delta(n,r)}-\underline{f}\le \left({4\over r}-{4\over r^2}\right)(\overline{f}-\underline{f}).
\end{equation}

\noindent Similar to the quadratic case (in Section \ref{secdeg2}), 
our new upper bound (\ref{iqcub}) implies the upper bound obtained by adding up (\ref{re9}) and (\ref{re11}). However, we do not find any example showing the upper bound (\ref{eqn2}) is tight. Thus, it is still an open question to show the tightness of the upper bound (\ref{eqn2}).

\ignore{
\noindent The following example illustrates that our new upper bound (\ref{iqcub}) can be not tight.

\begin{example}\cite[Example 3]{KLS13}\label{egdeg3}
Consider the cubic polynomial $f=x_1^3+x_2^3$. One can check that
$\overline{f}=1,\underline{f}={1\over 4}$ and
\begin{displaymath}
f_{\Delta(2,r)}=\left\{ \begin{array}{ll}
{1\over 4} & \text{if $r$ is even,}\\
{1\over 4}+{3\over 4r^2} & \text{if $r$ is odd.}
\end{array} \right.
\end{displaymath}
Moreover, for $r\ge 3$, by (\ref{fminrd}) one can check
\begin{displaymath}
f_{\Delta(2,r)}-f_{\min}^{(r-3)}= \left\{ \begin{array}{ll}
{1\over r-1}\left(\overline{f}-\underline{f}\right) & \text{if $r$ is even,}\\
{r^2+r-1\over r^2(r-1)}\left(\overline{f}-\underline{f}\right) & \text{if $r$ is odd.}
\end{array} \right.
\end{displaymath}
\noindent Hence, one can easily see (\ref{iqcub}) is not tight for this example.
\end{example}
}

\section{The square-free case}\label{secsqfree}

\noindent Consider the square-free (aka multilinear) polynomial $f=\sum_{I:I\subseteq [n],|I|=d}f_Ix^I\in\mathcal{H}_{n,d}$. We have the following 
result for the difference $f_{\Delta(n,r)}-f_{\min}^{(r-d)}$.
\begin{theorem}\label{thmsf}
For any square-free polynomial $f=\sum_{I:I\subseteq [n],|I|=d}f_Ix^I$ and $r\ge d$, one has
\begin{eqnarray}\label{fminsf}
f_{\Delta(n,r)}-f_{\min}^{(r-d)}\le \left({r^d\over r^{\underline{d}}}-1\right)\left(\overline{f}-\underline{f}\right).
\end{eqnarray}
\end{theorem}

\begin{proof}
From (\ref{fminrd}), one can easily check that
\begin{eqnarray*}
f_{\min}^{(r-d)}=\min_{\alpha\in I(n,r)}\sum_{I:I\subseteq [n],|I|=d}f_I{\alpha^I\over r^{\underline{d}}}={1\over r^{\underline{d}}}\min_{\alpha\in I(n,r)}f(\alpha).
\end{eqnarray*}
As a result, one can obtain
\begin{eqnarray*}
{f_{\min}^{(r-d)}\over f_{\Delta(n,r)}}={r^d\over r^{\underline{d}}}.
\end{eqnarray*}
For $d=1$, the result (\ref{fminsf}) is clear.

\noindent Now we assume $d\ge 2$.
Considering $\overline{f}\ge 0$ (as $f(e_i)=0$ for any $i\in[n]$), we obtain
\begin{eqnarray}\label{re18}
f_{\Delta(n,r)}-f_{\min}^{(r-d)}=\left(1-{r^d\over r^{\underline{d}}}\right)f_{\Delta(n,r)}\le \left(1-{r^d\over r^{\underline{d}}}\right)\underline{f} \le \left({r^d\over r^{\underline{d}}}-1\right)\left(\overline{f}-\underline{f}\right).
\end{eqnarray}
\qed
\end{proof}

\noindent The following example shows that our upper bound (\ref{fminsf}) can be tight.
\begin{example}\cite[Example 4]{KLS13}\label{egsf}
Consider the square-free polynomial $f=-x_1x_2$. One can check
$\overline{f}=0,$ $\underline{f}=-{1\over 4}$, and
\begin{displaymath}
f_{\Delta(2,r)}=\left\{ \begin{array}{ll}
-{1\over 4} & \text{if $r$ is even,}\\
-{1\over 4}+ {1\over 4r^2} & \text{if $r$ is odd.}
\end{array} \right.
\end{displaymath}
By (\ref{fminrd}), we have
\begin{displaymath}
f_{\Delta(2,r)}-f_{\min}^{(r-2)}=\left\{ \begin{array}{ll}
{1\over r-1}\left(\overline{f}-\underline{f}\right) & \text{if $r$ is even,}\\
\left({1\over r}+{1\over r^2}\right)\left(\overline{f}-\underline{f}\right) & \text{if $r$ is odd.}
\end{array} \right.
\end{displaymath}
\noindent For this example, the upper bound (\ref{fminsf}) is tight when $r$ is even.
In fact, from (\ref{re18}), one can easily see that the upper bound (\ref{fminsf}) is tight as long as $f_{\Delta(n,r)}=\underline{f}-\overline{f}$ holds.
\end{example}

\section{The general case}\label{secgene}

\noindent Now, we consider an arbitrary polynomial $f=\sum_{\beta\in I(n,d)}f_{\beta}x^{\beta}\in \mathcal{H}_{n,d}$. We need the following notation to formulate our result.
Consider the  univariate polynomial  $t^d-t^{\underline{d}}$ (in the variable $t$), which can be  written  as
\begin{equation}\label{tpd}
t^d-t^{\underline{d}}=\sum_{k=1}^{d-1}(-1)^{d-k-1}a_{d-k}t^{k},
\end{equation}

\noindent for some positive scalars $a_1,a_2,\dots,a_{d-1}$. Moreover, one can easily check that
\begin{eqnarray}\label{re19}
\sum_{k=1}^{d-1}a_{d-k}t^{k}=(t+d-1)^{\underline{d}}-t^d.
\end{eqnarray}
We can show the following
error bound for the range $f_{\Delta(n,r)}-f_{\min}^{(r-d)}$. 

\begin{theorem}\label{thmgeniq}
For any polynomial $f\in\mathcal{H}_{n,d}$ and $r\ge d$, one has
\begin{eqnarray}\label{re24}
f_{\Delta(n,r)}-f_{\min}^{(r-d)}\le {(r+d-1)^{\underline{d}}-r^d \over r^{\underline{d}}}{2d-1\choose d}d^d (\overline{f}-\underline{f}).
\end{eqnarray}
\end{theorem}

\noindent Note that when $f$ is quadratic, cubic or square-free, we have shown better upper bounds in Theorems \ref{thmqua}, \ref{thmcub} and \ref{thmsf}.\\

\noindent In the proof we will need the following Vandermonde-Chu identity (see \cite{PR01} for a proof, or alternatively use induction on $d\ge1$):
\begin{equation}\label{vcid}
(\sum_{i=1}^n x_i)^{\underline{d}}=\sum_{\alpha\in I(n,d)}{d!\over \alpha!}x^{\underline{\alpha}}\ \ \ \ \forall x\in\oR^n,
\end{equation}
which is an analogue of the multinomial theorem $(\sum_{i=1}^n x_i)^d=\sum_{\alpha\in I(n,d)}{d!\over \alpha!}x^{\alpha}.$

\noindent Now we prove Theorem \ref{thmgeniq}.

\begin{proof}{\em (of Theorem \ref{thmgeniq})} From (\ref{fminrd}), we have
\begin{eqnarray*}
{r^{\underline{d}}\over r^d}f_{\min}^{(r-d)}=\min_{\alpha\in I(n,r)}\left\{\sum_{\beta\in I(n,d)}f_{\beta}{\alpha^{\beta}\over r^d}-\sum_{\beta\in I(n,d)}f_{\beta}{\alpha^{\beta}-\alpha^{\underline{\beta}}\over r^d}\right\}.
\end{eqnarray*}
From this we obtain the inequality:

\begin{equation}\label{re5}
{r^{\underline{d}}\over r^d}f_{\min}^{(r-d)}\ge f_{\Delta(n,r)}-\max_{\alpha\in I(n,r)}\sum_{\beta\in I(n,d)}f_{\beta}{\alpha^{\beta}-\alpha^{\underline{\beta}}\over r^d}.
\end{equation}

\noindent We now focus on the summation $\sum_{\beta\in I(n,d)}f_{\beta}(\alpha^{\beta}-\alpha^{\underline{\beta}})$.

\noindent For any $\beta\in I(n,d)$ and $x\in\oR^n$, we can write the polynomial $x^{\beta}-x^{\underline{\beta}}$ as
\begin{equation}\label{xbeta}
x^{\beta}-x^{\underline{\beta}}=\sum_{\gamma:|\gamma|\le d-1}(-1)^{d-|\gamma|-1}c_{\gamma}^{\beta}x^{\gamma},
\end{equation}
\noindent for some nonnegative scalars $c_{\gamma}^{\beta}$ (which is an analogue of (\ref{tpd})).
We now claim that, for any fixed $k\in[d-1]$, the following identity holds:

\begin{eqnarray}\label{re6}
\sum_{\gamma\in I(n,k)}\sum_{\beta\in I(n,d)}{d!\over\beta!}(-1)^{d-|\gamma|-1}c_{\gamma}^{\beta}x^{\gamma}=(-1)^{d-k-1}a_{d-k}(\sum_{i=1}^nx_i)^{k}.
\end{eqnarray}

\noindent For this, observe that the polynomials at both sides of (\ref{re6}) are homogeneous of degree $k$. Hence (\ref{re6}) will follow if we can show that the equality holds after summing each side over $k\in[d-1]$. In other words, it suffices to show the identity:

$$\sum_{k=1}^{d-1}\sum_{\gamma\in I(n,k)}\sum_{\beta\in I(n,d)}{d!\over\beta!}(-1)^{d-|\gamma|-1}c_{\gamma}^{\beta}x^{\gamma}=\sum_{k=1}^{d-1}(-1)^{d-k-1}a_{d-k}(\sum_{i=1}^nx_i)^{k}.$$

\noindent By the definition of $a_{d-k}$ in (\ref{tpd}), the right side of the above equation is equal to $(\sum_{i=1}^nx_i)^{d}-(\sum_{i=1}^nx_i)^{\underline{d}}$. Hence, we only need to show

\begin{equation}\label{pfre1}
\sum_{k=1}^{d-1}\sum_{\gamma\in I(n,k)}\sum_{\beta\in I(n,d)}{d!\over\beta!}(-1)^{d-|\gamma|-1}c_{\gamma}^{\beta}x^{\gamma}=(\sum_{i=1}^nx_i)^{d}-(\sum_{i=1}^nx_i)^{\underline{d}}.
\end{equation}

\noindent Summing over  (\ref{xbeta}), we  obtain
$$\sum_{\beta\in I(n,d)}{d!\over \beta!}\left(x^{\beta}-x^{\underline{\beta}}\right)
=\sum_{\beta\in I(n,d)}\sum_{\gamma:|\gamma|\le d-1}{d!\over \beta!}(-1)^{d-|\gamma|-1}c_{\gamma}^{\beta}x^{\gamma}=\sum_{k=1}^{d-1}\sum_{\gamma\in I(n,k)}\sum_{\beta\in I(n,d)}{d!\over \beta!}(-1)^{d-|\gamma|-1}c_{\gamma}^{\beta}x^{\gamma}.$$

\noindent
We can now conclude the proof of (\ref{pfre1}) (and thus of (\ref{re6})). Indeed, by using the multinomial theorem and the Vandermonde-Chu identity (\ref{vcid}), we see that
the left-most side in the above relation is equal to  $(\sum_{i=1}^nx_i)^d-(\sum_{i=1}^nx_i)^{\underline{d}}.$


\noindent We partition $[d-1]$ as $[d-1]=I_{o}\cup I_{e}$, where $I_{o}:=\{k:k\in[d-1],\text{$d-k$ is odd}\}$ and $I_{e}:=\{k:k\in[d-1],\text{$d-k$ is even}\}$. Then, from (\ref{xbeta}), the summation $\sum_{\beta\in I(n,d)}f_{\beta}(\alpha^{\beta}-\alpha^{\underline{\beta}})$ becomes
\begin{eqnarray*}
&&\sum_{\beta\in I(n,d)}f_{\beta}(\alpha^{\beta}-\alpha^{\underline{\beta}})
=\sum_{\beta\in I(n,d)}f_{\beta}\sum_{\gamma:|\gamma|\le d-1}(-1)^{d-|\gamma|-1}c_{\gamma}^{\beta}\alpha^{\gamma}\\
&=&\sum_{k=1}^{d-1}\sum_{\gamma\in I(n,k)}\sum_{\beta\in I(n,d)}f_{\beta}(-1)^{d-|\gamma|-1}c_{\gamma}^{\beta}\alpha^{\gamma}\\
&\le& \left(\max_{\beta\in I(n,d)}f_{\beta}{\beta!\over d!}\right)\sum_{k\in I_{o}}\sum_{\gamma\in I(n,k)}\sum_{\beta\in I(n,d)}{d!\over\beta!}c_{\gamma}^{\beta}\alpha^{\gamma}
-\left(\min_{\beta\in I(n,d)}f_{\beta}{\beta!\over d!}\right)\sum_{k\in I_{e}}\sum_{\gamma\in I(n,k)}\sum_{\beta\in I(n,d)}{d!\over\beta!}c_{\gamma}^{\beta}\alpha^{\gamma}.
\end{eqnarray*}

\noindent By (\ref{re6}) we  obtain
\begin{eqnarray*}
\sum_{\beta\in I(n,d)}f_{\beta}(\alpha^{\beta}-\alpha^{\underline{\beta}})\le
\left( \max_{\beta\in I(n,d)}f_{\beta}{\beta!\over d!}\right) \sum_{k\in I_{o}}a_{d-k}(\sum_{i=1}^n\alpha_i)^{k}-
\left(\min_{\beta\in I(n,d)}f_{\beta}{\beta!\over d!}\right) \sum_{k\in I_{e}}a_{d-k}(\sum_{i=1}^n\alpha_i)^{k}.
\end{eqnarray*}

\noindent Combining with (\ref{re5}), we get
\begin{eqnarray*}
r^{\underline{d}}f_{\min}^{(r-d)}&\ge& r^df_{\Delta(n,r)}-\left(\max_{\beta\in I(n,d)}f_{\beta}{\beta!\over d!}\right) \sum_{k\in I_{o}}a_{d-k}r^{k}
+\left(\min_{\beta\in I(n,d)}f_{\beta}{\beta!\over d!}\right) \sum_{k\in I_{e}}a_{d-k}r^{k}.
\end{eqnarray*}

\noindent That is,
\begin{eqnarray*}
r^{\underline{d}}(f_{\Delta(n,r)}-f_{\min}^{(r-d)})
\le (r^{\underline{d}}-r^d)f_{\Delta(n,r)}+ \left(\max_{\beta\in I(n,d)}f_{\beta}{\beta!\over d!}\right) \sum_{k\in I_{o}}a_{d-k}r^{k}
-\left(\min_{\beta\in I(n,d)}f_{\beta}{\beta!\over d!}\right) \sum_{k\in I_{e}}a_{d-k}r^{k}.
\end{eqnarray*}

\noindent Since $r^{\underline{d}}-r^d=\sum_{k=1}^{d-1}(-1)^{d-k}a_{d-k}r^{k}$, we obtain
\begin{eqnarray*}
&&r^{\underline{d}}(f_{\Delta(n,r)}-f_{\min}^{(r-d)})\\
&\le& \sum_{k=1}^{d-1}(-1)^{d-k}a_{d-k}r^{k}f_{\Delta(n,r)}+ \left(\max_{\beta\in I(n,d)}f_{\beta}{\beta!\over d!}\right) \sum_{k\in I_{o}}a_{d-k}r^{k}
- \left(\min_{\beta\in I(n,d)}f_{\beta}{\beta!\over d!}\right) \sum_{k\in I_{e}}a_{d-k}r^{k}\\
&=& \left(\max_{\beta\in I(n,d)}f_{\beta}{\beta!\over d!}\right) \sum_{k\in I_{o}}a_{d-k}r^{k}+f_{\Delta(n,r)}\sum_{k\in I_{e}}a_{d-k}r^{k}
-\left(\min_{\beta\in I(n,d)}f_{\beta}{\beta!\over d!}\right)\sum_{k\in I_{e}}a_{d-k}r^{k}\\
&&-f_{\Delta(n,r)}\sum_{k\in I_{o}}a_{d-k}r^{k}.
\end{eqnarray*}

\noindent According to (\ref{reprop1}), one has $\min_{\beta\in I(n,d)}f_{\beta}{\beta!\over d!}\le f_{\Delta(n,r)}\le \max_{\beta\in I(n,d)}f_{\beta}{\beta!\over d!}$. Therefore, we have
\begin{eqnarray*}
r^{\underline{d}}(f_{\Delta(n,r)}-f_{\min}^{(r-d)})
\le \left(\max_{\beta\in I(n,d)}f_{\beta}{\beta!\over d!}-\min_{\beta\in I(n,d)}f_{\beta}{\beta!\over d!}\right)\sum_{k=1}^{d-1}a_{d-k}r^{k}.
\end{eqnarray*}

\noindent That is,
\begin{eqnarray*}
f_{\Delta(n,r)}-f_{\min}^{(r-d)}\le {\sum_{k=1}^{d-1}a_{d-k}r^{k} \over r^{\underline{d}}}\left(\max_{\beta\in I(n,d)}f_{\beta}{\beta!\over d!}-\min_{\beta\in I(n,d)}f_{\beta}{\beta!\over d!}\right).
\end{eqnarray*}

\noindent Finally, together with Theorem \ref{thmgene} and (\ref{re19}), we can conclude the result of Theorem \ref{thmgeniq}.\qed
\end{proof}

\noindent Now, we compare the following theorem by De Klerk et al. \cite{KLP06} with our new result (\ref{re24}).

\begin{theorem}\label{thmklpgene}\cite[Theorem 1.3]{KLP06}
Suppose $f\in \mathcal{H}_{n,d}$ and $r\ge d$. Then
\begin{equation}\label{re22}
\underline{f}-f_{\min}^{(r-d)}\le \left({r^d\over r^{\underline{d}}}-1\right){2d-1\choose d}d^d(\overline{f}-\underline{f}),
\end{equation}
\begin{equation}\label{re23}
f_{\Delta(n,r)}-\underline{f}\le \left(1-{r^{\underline{d}}\over r^{d}}\right){2d-1\choose d}d^d(\overline{f}-\underline{f}).
\end{equation}
\end{theorem}

\noindent 
By adding up (\ref{re22}) and (\ref{re23}), we obtain
\begin{equation}\label{re25}
f_{\Delta(n,r)}-f_{\min}^{(r-d)}\le \left({r^d\over r^{\underline{d}}}-{r^{\underline{d}}\over r^d}\right){2d-1\choose d}d^d(\overline{f}-\underline{f}).
\end{equation}
\begin{lemma}\label{claim1}
When $r$ is large enough, the upper bound (\ref{re24}) refines the upper bound (\ref{re25}).
\end{lemma}
\begin{proof}
It suffices to show that ${r^d\over r^{\underline{d}}}-{r^{\underline{d}}\over r^d}$ is larger than ${\sum_{k=1}^{d-1}a_{d-k}r^{k} \over r^{\underline{d}}}$ when $r$ is sufficiently large.
Since ${r^d\over r^{\underline{d}}}-{r^{\underline{d}}\over r^d}=(r^d-{(r^{\underline{d}})^2\over r^d})/r^{\underline{d}}$, we only need to compare $r^d-{(r^{\underline{d}})^2\over r^d}$ and $\sum_{k=1}^{d-1}a_{d-k}r^{k}$. For the term $r^d-{(r^{\underline{d}})^2\over r^d}$, one can check that the coefficient of $r^d$ is $0$ and the coefficient of $r^{d-1}$ is $2a_{1}>0$. On the other hand, in the summation ${\sum_{k=1}^{d-1}a_{d-k}r^{k}}$, the coefficient of $r^{d-1}$ is $a_{1}>0$. Therefore, when $r$ is sufficiently large, $r^d-{(r^{\underline{d}})^2\over r^d}$ is larger than $\sum_{k=1}^{d-1}a_{d-k}r^{k}$, by which we conclude the proof.\qed
\end{proof}

\noindent We illustrate the result in Lemma \ref{claim1} in the case of quartic polynomials.
\begin{example}\label{egdeg4}
Consider a polynomial $f\in\mathcal{H}_{n,4}$ written as
\begin{eqnarray*}
f&=&\sum_{i=1}^nf_ix_i^4+\sum_{i<j}\left(f_{ij}x_i^3x_j+g_{ij}x_i^2x_j^2+h_{ij}x_ix_j^3\right)+\sum_{i<j<k}(f_{ijk}x_i^2x_jx_k\\
&+&g_{ijk}x_ix_j^2x_k
+h_{ijk}x_ix_jx_k^2)
+\sum_{i<j<k<l}f_{ijkl}x_ix_jx_kx_l.
\end{eqnarray*}
\noindent In this case, (\ref{re24}) reads
\begin{equation}\label{re7}
f_{\Delta(n,r)}-f_{\min}^{(r-4)}\le {6r^2+11r+6\over (r-1)(r-2)(r-3)}{7\choose 4}4^4 (\overline{f}-\underline{f}),
\end{equation}
\noindent while (\ref{re25}) reads
\begin{equation}\label{re26}
f_{\Delta(n,r)}-f_{\min}^{(r-4)}\le {12r^2-58r+144-{193\over r}+{132\over r^2}-{36\over r^3}\over (r-1)(r-2)(r-3)}{7\choose 4}4^4 (\overline{f}-\underline{f}).
\end{equation}
\noindent One can check that (\ref{re7}) refines (\ref{re26}) when $r\ge 10$.
\end{example}

\begin{remark}
We now consider the convergence rate of the sequence
$$\alpha_r:={f_{\Delta(n,r)}-f_{\min}^{(r-d)}\over \overline{f}-\underline{f}}\ \ \ \ r=1,2,\dots$$

\noindent Suppose the degree of $f$ is fixed. By (\ref{re24}), we have $\alpha_r=O({1\over r})$. As in Example \ref{egsf}, $\alpha_r=\Omega({1\over r})$ holds, we can conclude that the dependence of $\alpha_r$ on $r$ in (\ref{re24}) is tight, in the sense that there does not exist any $\epsilon>0$ such that $\alpha_r=O({1\over r^{1+\epsilon}})$.

\noindent In \cite{KLS13}, De Klerk et al. consider the convergence rate of the sequence

$$\beta_r:={{f_{\Delta(n,r)}-\underline{f}}\over \overline{f}-\underline{f}}\ \ \ \ r=1,2,\dots$$

\noindent They consider several examples, and all of them satisfy $\beta_r=O({1\over r^2})$.
However, it is still an open question to determine the asymptotic convergence rate of $\beta_r$ in general.

\end{remark}

\section*{Acknowledgements}
\noindent The author is grateful to M. Laurent and E. de Klerk for useful discussions and for their help to improve the presentation of this paper. The author also thanks the anonymous reviewers for useful remarks.

\end{document}